\documentclass[12pt]{article}
\usepackage[margin=0.8in]{geometry}
\usepackage{amsmath,amssymb,amsthm,amsfonts,color,hyperref,graphicx,subfigure,authblk}
\usepackage{indentfirst}
\usepackage{comment}

\makeatletter
\def\blfootnote{\gdef\@thefnmark{}\@footnotetext}
\makeatother

\hypersetup{colorlinks,
linkcolor=blue} 

\theoremstyle{definition}
\newtheorem{thm}{Theorem}[section]
\newtheorem{lem}[thm]{Lemma}
\newtheorem{rem}[thm]{Remark}
\newtheorem{defi}[thm]{Definition}
\newtheorem{prop}[thm]{Proposition}

\newtheorem{question}{Question}

\numberwithin{equation}{section}

\providecommand{\keywords}[1]
{
  \textbf{\text{Keywords: }} #1
}
\newcommand{\enabstractname}{Abstract}
\newenvironment{enabstract}{
    \par\small
    \noindent\mbox{}\hfill{\bfseries \enabstractname}\hfill\mbox{}\par
    \vskip 2.5ex}{\par\vskip 2.5ex} 

\title{On Geodesic Rays of Newtonian Gravitational Systems}
\author[]{Putian Yang}
\author[]{Shiqing Zhang\thanks{Corresponding author: zhangshiqing@scu.edu.cn}}

\affil[]{School of Mathematics, Sichuan University}

\date{}

\begin{document}
\maketitle

\blfootnote{\textup{2010} \textit{Mathematics Subject Classification}: 34B15, 34C25, 70F10, 70F16, 70G75}

\begin{enabstract}
 In this paper, we focus on the set of geodesics rays of the Newtonian $N$-body problem.
 We find that the limits of geodesic rays are also geodesic rays, hence they are not dense in the space of initial conditions. As a result, 
 there are many motions whose domain is the half real line has non-negative total energy, and they are not geodesic rays, 
 the set of such motions has positive measure,
 we think this set gives a large space to accommodate many solutions with ``bad" properties.
 As an application of weak KAM theory for $N$-body problem, we give a brief proof of the existence of complete parabolic orbit starting from
  any given initial position. 
\end{enabstract}

\tableofcontents
\newpage

\section{Introduction}\label{Introduction}
In 1922, Chazy \cite{ASENS_1922_3_39__29_0} published a long famous paper on the three body problem,
he claimed that there are seven kinds of motions based on their final evolutions, among which the parabolic and hyperbolic types attract tremendous interests,
since these types and mixed ones 
have  interesting physical features.

In this paper we suppose that there are $N$ punctual bodies with masses $m_1,\dots, m_N>0$. 
They move in the given Euclidean space 
$E=\mathbb{R}^d, d\geq 2$, 
with $(\cdot,\cdot) $ the standard inner product of $E$ and $\vert\cdot\vert$ is its induced norm. We call $E^N$  the configuration space of the $N$-body problem,
it is a linear space with 
mass scalar product which is denoted by $g_m$ and is defined as
\begin{equation}
\langle x,y\rangle=\sum_{i=1}^{N}m_{i}(x_{i},y_{i})
\end{equation}
for $x=(x_{1},\ldots,x_{N}), y=(y_{1},\ldots,y_{N})\in E^{N}$.
The moment of inertia of $x$ is $I(x)=\langle x,x\rangle$ and the norm of $x$ is $\|x\|=I(x)^{1/2}.$

For $N$ bodies with positions $x_1(t),\dots, x_N(t)$,  the parabolic orbits, hyperbolic orbits and their mixed types have the definitions
 simlilar to that in  \cite{ASENS_1922_3_39__29_0}.

 \begin{defi}
    For $x(t) = (x_1(t),\dots, x_N(t))\in E^N$, as $t\rightarrow +\infty$, the $N$ bodies system is said to be 
\begin{itemize}
    \item   completely parabolic or simply parabolic, if each body $x_i$ goes to infinity and $\dot{x}_i\rightarrow 0\in E$;
    \item   hyperbolic, if velocity of each body $\dot{x}_i\rightarrow a_i\in E, a_i\neq a_j$, $i,j = 1,\dots, N$;
    \item   hyperbolic-parabolic or partially hyperbolic, if velocity of each body $\dot{x}_i\rightarrow a_i\in E$, there are some  $a_i$ that are equal, and not all $a_i$ be zero.
\end{itemize}
\end{defi}

We say that the motion of a $N$-body system is called expansive if all mutual distances $r_{ij}(t) = \vert x_i(t) - x_j(t)\vert \rightarrow +\infty$ as $t\rightarrow+\infty$. 
On the other hand, a normalized configuration $a\in E^N$ is called the limit shape of a solution $x(t)$, if there are constants $\alpha,\kappa>0$  such that
\begin{equation*}
   \lim_{t\rightarrow+\infty}\frac{x(t)}{t^\alpha} = \kappa a.
\end{equation*}

Denote by $\Omega$  the subset of configurations without collision, i.e., 
\begin{equation*}
    \Omega=\left\{x\in E^{N}\mid x_{i}\neq x_{j},  \forall i,j \text{ s.t. } 1\leq i<j\leq N\right\},
\end{equation*}
 we also denote by $\Delta=E^{N}\setminus\Omega$ the collision set  and by $TE^N\cong E^N\times E^N$ the space of initial conditions, 
    in particular an initial position being in $\Delta$ means that its motion is an ejection emerging from collision.

We equipe $\Omega$ with the Jacobi-Maupertuis metric, then $\Omega$ is an incomplete Riemannian manifold, hence we can consider geodesic rays starting from 
given points in $\Omega$. 
Apparently any expansive motion has non-negative total energy and Chazy also proved that for the 
three body problem, all the three cases are are expansive.
Burgos and Maderna
proved in  \cite{Burgos2022} that  geodesic rays of $N$-body problem all belong to the above three types, and
all geodesic rays are expansive, we present this result in Theorem \ref{free time minmizer with h geq 0, all r_ij tend to infinity}.
We mainly discuss geodesic rays in this paper. 

Chazy \cite{ASENS_1922_3_39__29_0} had verified that the hyperbolic motions are stable in three-body problem, in other words,  
the set of initial conditions leading to hyperbolic motions is open in $T\Omega$,
i.e., a solution that is near a hyperbolic soluton is also hyperbolic. 
Besides Chazy's great work, during the last century, there are some other early literatures on the initial conditons of expansive motions,
for example, 
Saari \cite{SAARI1984300}, Saari and Hulkower \cite{SAARI198127} investigated the manifold structures of some sets of initial conditions of expanding orbits and their connection to collisions.
Another example is from Saari \cite{10.2307/1995609}, he found that when the total energy is positive, 
the growth order of the minimal distance between particles is either $t^{2/3}$ or $t$.
In this paper, we prove that if both the dimension $d$ of the space $E$ and the number of the bodies $N$ are greater than $2$, 
the set of all geodesic rays of $N$-body problems is closed in the topology of $TE^N$. Through out this paper we denote by $GR$ the set of initial conditions leading to geodesic rays.
We apply the methods from Maderna and Venturelli \cite{10.4007/annals.2020.192.2.5} 
to get our main result as the following:
\begin{thm}\label{The set of geodesic rays is closed}
    In the Newtonian gravitational $N$-body system, the set of initial conditions giving rise to geodesic rays is closed.
\end{thm}

The existence of expansive motions were unclear in 20th century, 
for instance, whether there will be a parabolic eject from a given initial configuration or to a given limit configuration.
It is rather difficult for the calculus method to solve these existence problems.

Nevertheless, in recent years, weak KAM theory and some related topological technics have attracted considerable attentions in the community of celestial mechanics,
and by using these tools there have been extensive outstanding studies in the existence of Newtonian expansive motions, 
among them we must emphasize here several researches done by excellent mathematicians in the field of dynamical systems, and we summarize them 
 as
Theorem \ref{exist of geodesic motions}. 

In 2009, Maderna and Venturelli \cite{Maderna2009} discovered that given a central configuration $x_1\in E^N$ that 
absolutely minimizes the normalized potential $\tilde{U} = I^{1/2}U$, 
and
given any $x_0$ as a initial point in $E^N$, there exist a motion $x(t)$ 
such that $x(0) = x_0$ and $x_1$ is its limit shape in the sense that
\begin{equation*}
    x(t) = \lambda x_1 t^{2/3} + o(t^{2/3}) 
\end{equation*}
 for some constant $\lambda>0$.

 In 2013, Luz and Maderna \cite{daluz_maderna_2014} gives another beautiful proof of the existence of parabolic motions.
 They found that the parabolic motions can emerge from any initial point in $\Omega$, and the motion is asymptotic to the $set$ of central configuration,
 in other words, for an arbitrary given position $x_0 \in \Omega$, there are velocities of the $N$ particles 
 such that the motion starting from the chosen initial conditions is fully parabolic.
 
 For the case that the parabolic motions starting from the collision set and asymptotic to a given central 
 configuration, see \cite{Percino2014}. Based on \cite{Burgos2022},
 we in our paper, give another brief proof of the existence of parabolic motions.
 
 In 2020, Maderna and Venturelli published another extrordinary paper \cite{10.4007/annals.2020.192.2.5}, he proved the existence of hyperbolic motions
 that given any initial configuration $x_0\in E^N$, any normalized configuration without collisons
 $x_1\in\Omega$, and any total energy $h>0$, there exist a hyperbolic motion $x(t)= \sqrt{2h}x_1t +o(t)$ starting from
 $x_0$ and has  $x_1$ as its limit shape.

 Recently, Burgos \cite{burgos2022existence} proved the existence of a partially hyperbolic motion. 
 That is, in an Euclidean space $E=\mathbb{R}^d$ with dimension $d\geq 2$,
 for any given energy level $h>0$, any configuration $x_0\in\Omega$, there exists a partially hyperbolic motion $x(t) = at + o(t)$ 
 such that $a\in\Delta-\{0\}$. Burgos's arguement relies on an important work of Chazy in  \cite{ASENS_1922_3_39__29_0}, that the limit shape of hyperbolic motion depends contiunously on the initial condition. 
 There is a natural question.
 
 \begin{question}
     Given an arbitrary $a\in\Delta\setminus\{0\}$, is there a motion that has $a$ to be its limit shape?
 \end{question}

 For brevity, we summarize here some of their conclusions as the following theorem, note that if $x_0\in\Delta$,
 it means that the motion is an ejection from collision and $x$ is continous at $t=0$, and $x(t)\in\Omega$ is a maximal 
solution for $t>0$.
\begin{thm}\label{exist of geodesic motions}
For the Newtonian $N$-body problem, given any intial configuration $x_0\in E^N$ and any energy constant $h\geq 0$,
there are choices of initial velocities that determine the motion to be a geodesic rays, they have the form
\begin{equation}
   x(t)=at+O(t^{2/3})\text{ as }t\rightarrow+\infty,
\end{equation}
and $x(t)$ is hyperbolic if $a\in\Omega$,
parabolic  if $a=0$,
 partially hyperbolic  if $a\in\Delta\setminus\{0\}$.

If $h=0$, it is completely parabolic, if $h> 0$, it is hyperbolic or partially hyperbolic.
\end{thm}

Now we state an important theorem as following that will be used in proof of our main result.   
See Burgos and Maderna \cite{Burgos2022} for the proof of the theorem.
\begin{thm}\label{free time minmizer with h geq 0, all r_ij tend to infinity}
For any geodesic ray $\gamma:[0,+\infty)\rightarrow\Omega$ with energy constant $h\geq 0$, 
the motion is expansive, more precisely, $r_{ij}\thickapprox t^{2/3}$ or $r_{ij}\thickapprox t$.
\end{thm}
In this paper we do not directly prove Theorem \ref{The set of geodesic rays is closed}, 
on the contrary we use an important equivalence between free time minimizers and geodesic rays, please see Theorem \ref{ free time minimizer equivalent to geodesic ray}, 
and we use the methods concerned with free time minimizers.

Another interesting discussion of our paper is about the non-geodesic motions. It is known that many 
solutions without singularity are not geodesic, for example, the periodic solutions and the superhyperbolic motions. 
When considering the case of non-negative total energy, the set of solutions of Newtonian gravitational system 
that is not geodesic $may$ includes many relatively bad solution such as superhyperbolic motions and oscillatory motions.
Based on Theorem \ref{The set of geodesic rays is closed}, we have a natural result for that set, 
that at least for some small numbers of bodies, such set has positive measure, 
this assertion gives us a lot of imagination, we present it as the following proposition:
\begin{prop}\label{wide existence of non-geodesic motions}
    For $d=2,3$,  $N= 2,3,4$, and arbitrary $N$ positive point masses, there are many initial conditions leading to solutions which do exist in $[0,+\infty)$,
     they are not geodesic ray.
      The set of such initial data near the singularity set has non-zero measure.
 \end{prop}

 \begin{proof}

 According to Theorem \ref{The set of geodesic rays is closed}, in $GR$ there cannot be a sequence with a limit leading to motions with singularities. 
 For any point $p$ in the set of initial conditions $NC$ that leads to the singularity,
 there must be an open ball $B(p,r_p)$ which contains no point in $GR$. On the other hand, Saari \cite{SAARI197780} 
 found that $NC$ is of measure zero in case $d=2,3$ and $N=2,3,4$,
 hence $\cup_{p\in NC} B(p,r_p) \setminus NC$ has non-zero measure, it is the desired set.

 \end{proof}
Proposition \ref{wide existence of non-geodesic motions} implies that even for the positive total energy,
the solutions near the singularity set suffer no singularity and do not minimize the Lagrangian action in such a strong sense.
These solutions are close to the singularity set so that their behaviors may not look good.

We organize the paper as follows. The second section is devoted to introducing the background material, some results 
given by variational method and notations.  
In particular,
subsection \ref{Superhyperbolic motions} introduces superhyperbolic motions, 
and
subsection \ref{Parabolic motions starting from collison set}
gives a brief proof of the existence of parabolic motions.
In the third section, our main result is proved, and a question is proposed.

\section{Variational Setting and Some Results}

 The dynamics of the $N$ bodies are ruled by the Newton's universal gravitation law, that is
\begin{equation}\label{Newton motion equation}
    m_i\ddot{x}_i=\nabla_{x_i} U(x),\text{ for } 1\leq i\leq N,
  \end{equation}
 the Newtonian potential
\begin{equation}
   V(x) =  -U(x): = -\sum_{1\leq i<j \leq N}\frac{m_i m_j}{\vert x_i - x_j\vert},
\end{equation}
where we take suitable units such that the gravitation constant $G=1$.

A configuration that absolutely minimizes $I^{1/2}U$ is also called a minimizing central configuration.

The simplest hyperbolic motions or parabolic motions are those homothetic type, in these situations, the limiting configurations are central configurations. 
These constructions limit the choices of initial conditions giving rise to these kinds of solutions, since the central configurations up to similarity
are finite, at least for $N=2,3,4,5$, see Hampton and Moeckel \cite{Hampton2006} for the spatial four body problem, Albouy and Kaloshin \cite{10.2307/23234173} for the planar five body case
and Hampton and Jensen \cite{Hampton2011} the spatial five body problem.
For more general $N$-body problem, Painlev\'{e}-Wintner conjecture that the number of central configurations up to similarity is finite.
However, recent studies give us a broad view that the initial conditions leading to our desired orbits are not necessarily central configurations
and they do exist almost everywhere.

The Lagrangian $L$ on $E^N \times E^N$ is defined as
  \begin{equation}
    L(x,v)=\frac{1}{2}\|v\|^{2}+U(x).
  \end{equation}
  According to  the famous Hamilton's principle, a solution of the Newton's equation joining two arbitrary given $x,y\in E^N$ must be a critical point of the Lagrangian action, and more precisely
  a minimizer of Lagrangian action among space of absolutely continuous curves joining $x$ and $y$ with  fixed time interval. 
  The existence of such minimizers  was already confirmed by Tonelli's theorem, see F. Clarke  \cite[p321, Theorem 16.2]{clarke2013functional}.
  We next give this approach explicitly.
  
  Define the set of absolutely curves
  
  \begin{equation}
     AC  (x,y,T):=\left\{\gamma:[0,T]\rightarrow E^N \text{ is absolutely continuous, } \gamma(0)=x,\gamma(T)=y\right\}
   \end{equation}
   and
   \begin{equation}
     AC  (x,y)=\cup_{T>0}   AC  (x,y,T),
   \end{equation}
   since without loss of generality, we can always assume each curve starting at $0$ ending at some $T>0$ for $T$.

The Hamiltonian $H$ of the Lagrangian is the first integral of the solutions to the Newton's equation, and 
$h$ is the total energy constant.
The Hamilton-Jacobi equation is 
\begin{equation}\label{HJ equation}
    H(x,d_x u)= \frac{1}{2}\|d_x u\|_* -U(x) =h,
\end{equation}
where the norm $\|\cdot\|_*$ is the norm of $(E^*)^N$, for $p\in (E^*)^N$,
\begin{equation}
\|p\|_* = \sum_{i=1}^N m_i^{-1}\vert p_i\vert^2.
\end{equation}
The Lagrangian action of $\gamma\in  AC  (x,y)$ with energy constant $h\geq 0$ is

\begin{equation}
    A_h (\gamma)=\int_0^TL(\gamma,\dot{\gamma})+hdt= \int_0^T \frac{1}{2}\|\dot{\gamma}\|^{2}+U(\gamma)+hdt
  \end{equation}
when $h=0$, we simply write $A$. It is well defined since $\Omega$ is connected, and for all $x,y\in\Omega$ there is a smooth $\gamma\in  AC  (x,y,T)$ and $U$ is bounded 
in the compact set $\gamma\vert_{[0,T]}$ hence $A_h (\gamma)<\infty$ and $0\leq A_h\not\equiv+\infty$.
  
  However, there was a big obstacle for the calculus of variations that under Newton potential there will be trajectory with isolated collisions and its Lagrangian action is finite, 
  this discovery is due to Poincar{\'e} \cite{zbMATH02677123}, leading to a question that a minimizer may not be a true motion.
  But mathematicians did not worry about it long since Marchal   \cite{Marchal2002} have made a great breakthrough,  making sure 
  in case of more than 2 dimensional, the minimization process 
  confirms the absence of collision in the classical Newtonian $N$-body problem, see more in Chenciner \cite{zbMATH01789893}.
  In 1-D case, Yu and Zhang \cite{Yu2018} has obtained important results under some conditions.

  \begin{thm}\label{Marchal: minimizer has no collison}
    (Marchal \cite{Marchal2002} Chenciner \cite{zbMATH01789893})
    If $\gamma\in AC (x,y)$ is defined over some interval $[0,T]\subset\mathbb{R}$ 
    and it minimizes $A$ over $ AC (x,y,T)$, then $\gamma$ experiences no collison for all $t\in(0,T)$.
\end{thm}
Therefore a minimizer of the action is a solution of the Newtonian equation \ref{Newton motion equation} in $(0,T)$ and a true motion of Newtonian systems.

As investigated by Saari \cite{10.2307/1995752,SAARI197780}, we know that for the planar or three-dimensional  two, three or 
four body problems under Newtonian potential, in the sense of Lebesgue, almost all $(x,v)\in T\Omega$  leading to solutions confirm no singularities and exist for $[0,+\infty)$.
We denote  $NC\subset TE^N$ as the set of initial coniditions leading to motions that suffer singularities at some definite times.

Let $x,y\in E^N$ be arbitrary, we define the fixed time action potential for time $T$:
\begin{equation*}
    \phi(x,y,T)=\inf \left\{ A(\gamma)\big\vert\gamma\in AC (x,y,T) \right\},
\end{equation*}
and the free time action potential $ \phi_h (\cdot,\cdot)$:
\begin{equation*}
    \phi_h (x,y)=\inf\{A_h (\gamma)\mid\gamma\in  AC  (x,y)\}.
\end{equation*}
A curve $\gamma_0\in  AC  (x,y,T)$ is called a fixed-time minimizer of $A$ if it minimizes $A$ in $ AC  (x,y,T)$.
Similarly, a curve $\gamma_0\in  AC  (x,y)$ is called a free time minimizer of $A_h$ if it minimizes $A_h$ in $ AC  (x,y)$, 
the existence of free time minimizers is proved, see  \cite[Theorem 3.1]{daluz_maderna_2014} and  \cite[Lemma 4.2]{10.4007/annals.2020.192.2.5}.
It is obvious that a free time minimizer is a fixed time minimizer hence experiences no collision in the middle of the interval.

\begin{thm}\label{free time minimizer of phi_h exist}
    In a Newtonian $N$-bodies system with energy constant $h\geq 0$, free time minimizing curve of its action potential $\phi_h(x,y)$ does exist for any distinct $x,y\in E^{N}$.
\end{thm}
A curve $\gamma$ defined in $[0,+\infty)$ is called a free time minimizer if it is a free time minimizer 
of the Lagrangian action when restricted in any compact subinterval $[a,b]\subset [0,+\infty)$,
that is, $\gamma_{[a,b]}$ minimizes Lagrangian over $ AC (\gamma(a),\gamma(b))$.
Following the idea of Marchal, if $\gamma\in AC (x,y,T)$ minimizes $A_h$ over $ AC (x,y)$, then it must minimize $A_h$
over $ AC (x,y,T)$, hence minimizes $A$ over $ AC (x,y,T)$ since $A_h=A+hT$.
Therefore, if $\gamma\in AC (x,y)$ minimizes $A_h$, then $\gamma\vert_{(0,T)}\subset\Omega$.

We have to mention that the methods of free time action minimizing is not the only way to find those expansive motions.
For instance,
Maderna and Venturelli \cite{Maderna2009} did not assert the solution $x(t)$ approximating to a given central configuration to be free time minimizing.
There are some other literatures as examples, see Duignan, Moeckel, Montgomery and Yu \cite{Duignan2020} for bi-hyperbolic motions,
and Zhang \cite{Zhang2012} for restricted three body problem. On the other hand, Hu, Ou and Yu  \cite{Hu2021} suggested that the results in  \cite{Hu2021}  could
be useful in the application of non-action minimization methods in the Newtonian
$N$-body problem.
Therefore it is a natural question that whether a hyperbolic, a parabolic or a partial hyperbolic motion is necessarily geodesic,
this is an unsolved problem.

The Jacobi-Maupertuis metric is defined as $j_h = 2(h+U)g_m$.
A geodesic ray of the $N$-body system  is an arclength parameterized geodesic $\gamma:[0,+\infty) \rightarrow E^N$ such that all of its restrictions to compact subintervals are minimizing geodesics.

The total energy constant of a given motion $x(t)$ is $h=\frac{1}{2}\|\dot{x}(t)\|-U(x(t))$, which implies that the motion must be confined to the Hill's region:
\begin{equation}
    \Omega_h=\{x \in \Omega \mid U(x) \geqq-h\} .
\end{equation}
We see that $\Omega_h=\Omega$ if and only if $h \geqq 0$, 
the geodesics of the Jacobi-Maupertuis metric and 
free time minimizers of $A_h$ are equivalent, roughly speaking, this is probably because 
for any $\gamma\in AC (x,y)$,
\begin{align*}
    \int_s^t \|\dot{\gamma}\|_{j_h} d\tau & = \int_s^t \sqrt{2(h+U)}\|\dot{\gamma}\| d\tau = \int_s^t \sqrt{2}\left(\frac{1}{2}\|\dot{\gamma}\|^2 -U+U \right)^{1/2}\|\dot{\gamma}\|d\tau \\
       & = \int_s^t \|\dot{\gamma}\|^2 d\tau = \int_s^t \frac{1}{2}\|\dot{\gamma}\|^2 + U + \frac{1}{2}\|\dot{\gamma}\|^2 -U d\tau \\
       & = \int_s^t L(\gamma,\dot{\gamma}) +h d\tau = A_h\left(\gamma\vert_{[s,t]}\right).
\end{align*}
The rigorous proof is given by Burgos and Maderna \cite{Burgos2022}.

\begin{thm}\label{ free time minimizer equivalent to geodesic ray}
    After a suitable reparameterization, a geodesic ray of the metric $j_h$ is a free time minimizer of $A_h$,
    and a free time minimizer of $A_h$ is a geodesic ray of the metric.
\end{thm}

Further more, we denote by $d_h$ the Riemannian distance induced by the metric $j_h$,
 Maderna and Venturelli \cite{10.4007/annals.2020.192.2.5} have pointed out that the completion of $(\Omega, d_h)$ is exactly $(E^N,\phi_h)$.
 $\phi_h$ is a distance,
 it is verified by Maderna \cite{maderna_2012} for $h=0$, but it also holds for cases of $h\geq 0$. We owe the proof of the following proposition to Maderna. 
\begin{prop}(Maderna \cite{maderna_2012})
For any $h\geq 0$, $ \phi_h (\cdot,\cdot)$ is a distance in $E^N$.
\end{prop}

\begin{proof}
  We first prove that $ \phi_h $ satiesfies trianglular inequality. For any $x,y,z\in E^N$, any $\gamma_{1}\in  AC  (x,y)$ and $\gamma_{2}\in  AC  (y,x)$ defined in $[0,T_{1}]$ and $[0,T_{2}]$ respectively, we set 
\begin{equation}
  \gamma(t):=
  \begin{cases}
    \gamma_{1}(t) &0\leq t\leq T_{1}\\
    \gamma_{2}(t-T_{1}) &T_{1}\leq t\leq T_{1}+T_{2}
  \end{cases},
\end{equation}
thus \begin{equation}
   \phi_h (x,z)=\inf\left\{A_h (\eta)\mid\eta\in  AC  (x,z)\right\}\leq A_h (\gamma)\leq A_h (\gamma_{1})+A_h (\gamma_{2}).
\end{equation}
Since $\gamma_{1},\gamma_{2}$ are arbitrary, we have $ \phi_h (x,z)\leq \phi_h (x,y)+ \phi_h (y,z)$.

Second, we verify that $ \phi_h (x,y)=0$ makes $x=y$.
For any $x=(x_{1},\ldots,x_{N}),y=(y_{1},\ldots,y_{N})\in \Omega$ and $\gamma=(\gamma_{1},\ldots,\gamma_{N})\in  AC  (x,y)$ defined over $[0,T]$.
For any $1\leq i\leq N$,
\begin{equation*}
    \vert x_i -y_i\vert = \vert \gamma_i(T) - \gamma_i(0)\vert \leq \int_0^T\vert\dot{\gamma}_i\vert dt\leq\sqrt{T}\left( \int_0^T \vert \dot{\gamma}_i \vert^2 \right)^{1/2}.
\end{equation*}
Therefore 
\begin{equation*}
    A_h(\gamma)\geq \int_0^T \frac{1}{2}\|\dot{\gamma}\|^2 dt = \frac{1}{2} \sum m_i \int_0^T\vert \dot{\gamma}_i \vert^2 dt \geq \frac{1}{2T}\sum m_i\vert x_i - y_i \vert^2.
\end{equation*}
So eventually $ \phi_h (x,y)=0$ indicates $x=y$.

It is not difficult to see $ \phi_h (x,x)=0$ and $\phi_h(x,y) = \phi_h(y,x)$.
\end{proof}

We must mention an important estimate for $\phi_h(\cdot,\cdot)$ such that we can apply Ascoli's theorem to get our desired solution. 
This estimate is given by
Maderna \cite{maderna_2012}.
\begin{thm}\label{bound for phi(x,y,t)}(Maderna \cite{maderna_2012})
    For a Newtonian $N$-body system in a configuration space $E^{N}$, 
    the masses of the particles and the number $N$ determine two positive constants
    $C_1,C_2$ such that 
    \begin{equation}
        \phi(x,y,t)\leq C_1\frac{l^2}{t}+C_2\frac{t}{l}
    \end{equation}
    for any $x,y\in E^{N}$, any $l>\|x-y\|$ and $t>0$.
\end{thm}

The following theorem is about a uniform bound for a sequence of action potential, 
it is a tiny extension of  \cite[Theorem 2.11]{10.4007/annals.2020.192.2.5} .
\begin{thm}\label{uniform bound for phi_hn}
    If  $\{h_{n}\}$  is  a sequence of bounded non-negative energy constants. There exist constants $\alpha>0$ and $\beta> 0$ such that,  
    \begin{equation}
        \phi_{h_{n}}(x,y)\leq \mu(\|x-y\|),
    \end{equation}
    where $ \mu(\|x-y\|)  = \left(\alpha\|x-y\|+\beta\|x-y\|^2\right)^{1/2}$ and
 $\alpha,\beta$ only depend on the bound of $h_{n}$, the number of bodies $N$ and their masses.
\end{thm}

\begin{proof}
    Suppose $h_n$ are bounded by a positive number $M$, since $\phi_{h_{n}}(x,y)=\inf_{t>0}\left\{\phi(x,y,t)+h_n t\right\}$, 
    and follows from Theorem \ref{bound for phi(x,y,t)} we first estimate the upper bound for $\phi+h_n t$:
    $$
    \phi(x,y,t)+h_n t\leq \frac{C_1l^2}{t}+(\frac{C_2}{l}+h_n)t,
    $$for all $t>0,l>\|x-y\|$,
    now consider the right handside of the above inequality as a function of $t$, we calculate its minimal value to be 
    $2\sqrt{C_1 l^2\left(\frac{C_2}{l}+h_n\right)}=\left( 4C_1 C_2 l +4C_1 h_n l^2 \right)^{1/2}$, thus if we
    let $\alpha=4C_1C_2, \beta = 4C_1 M $, then by $h_n\leq M$ we have 
    $$
    \phi_{h_n}(x,y)\leq \phi(x,y,t)+h_n t \leq \left(\alpha l + \beta l^2\right)^{1/2}
    $$ holds for all $l>\|x-y\|$, 
    then let $l\rightarrow \|x-y\|$ to justify our claim.
\end{proof}

\subsection{The calibrating curves}

The purpose of this subsection is to present certain minimizers which are called calibrating curves, they relate the action potentials and the viscosity subsolutions of
the corresponding Hamilton-Jacobi equation.
  Proposition \ref{Maderna: exist calibrating curve on [0,+infty)} claim the existence of such curves, it was proved by Maderna and Venturelli  \cite[Theorem 3.2, p525]{10.4007/annals.2020.192.2.5}
for $h>0$,  the paper  \cite{10.4007/annals.2020.192.2.5} implies that the result also holds in case of zero energy. For convenience, we present here the proof of Proposition
 \ref{Maderna: exist calibrating curve on [0,+infty)} and owe it to Maderna and Venturelli.
Calibrating curves is quite essential through out of our paper.

\begin{defi}
    A absolutely continuous curve $\gamma:I \rightarrow E^N$ in some interval $I$ is called a calibrating curve of a function $u\in C(E^N)$ if 
    \begin{equation}
        u\left(\gamma(t_2)\right) - u\left(\gamma(t_1)\right) = A_h\left(\gamma\big\vert_{[t_1,t_2]}\right)
    \end{equation}
    for any closed compact subinterval $[t_1,t_2]\subset I$. We also call $\gamma$ $h$-calibrates $u$ in $I$.
\end{defi}

Before proving the proposition, we need 
von Zeipel's theorem, see von Zeipel \cite{von1908singularites} and McGehee \cite{zbMATH04009869}.
\begin{thm}\label{von Zeipel}
 If the solution $x(t)$ has a singularity at $t^*$, 
 then $\lim_{t\rightarrow t^*} I(x(t)) = I^*$ exists and lies in $[0, +\infty]$.
 In particular, in the case of  $\lim_{t\rightarrow t^*}I(x(t)) = I^*<+\infty$, 
 there exists a $x^* \in \Delta$ such that $\lim_{t \rightarrow t^*}x(t) = x^*$, i.e., $x(t)$ suffers a collision.
\end{thm}
In the case of $I^* = +\infty$, the particles of $x(t)$ is called experiencing a pseudo-collison.

\begin{prop}\label{Maderna: exist calibrating curve on [0,+infty)}
    For $h_n\geq 0$ and $h_n\rightarrow h\geq 0$, if $u\in C^0(E^N)$ and $u(x)=\lim_{n\rightarrow\infty}\left(\phi_{h_n}(0,p_n) - \phi_{h_n}(x,p_n)\right)$ 
    for some sequence $p_n\in E^{N}$ 
    with $\|p_n\|\rightarrow+\infty$. 
    Then for any initial configuration $x\in E^N$, there exist a calibrating curve $\gamma$  defined over $[0,+\infty)$ with $\gamma(0)=x$.
    i.e., 
    \begin{equation*}
        u(\gamma(t)) - u(x)= A_h \left(\gamma\big\vert_{[0,t]}\right)
    \end{equation*}for all $t>0$.
\end{prop}

\begin{proof}
    We first show that for any $r>0$ there exist some $y_r\in E^N$ and a curve $\gamma_r\in AC (x,y_r)$ defined in a finite interval, 
    such that $\|x-y_r\|=r$ and $u(y_r) - u(x)= A_h (\gamma_r)$.
    
    Denote $\phi_{h_n}(\cdot,p_n)$ by $v_n(\cdot)$ as a function  in $E^N$, then it is continuous since $\phi_{h_n}$ is a distance.
    According to Theorem \ref{free time minimizer of phi_h exist}, the action $\phi_{h_n}(x,p_n)$ attains its minimum in $ AC (x,p_n)$ 
    at some $\gamma_n$ defined over $[0,t_n]$ with $\gamma_n(t_n) = p_n$, then for any $t\in[0, t_n]$,
    \begin{equation*}
        v_n(\gamma(t))-v_n(p_n) = \phi_{h_n}(\gamma(t),p_n)= A_{h_n} \left(\gamma_n\vert_{[t,t_n]}\right),
    \end{equation*}
    which implies that for any subinterval $[t_1, t_2]\subset [0, t_n)$, 
    \begin{align*}
        &v_n(\gamma(t_1)) - v_n(\gamma(t_2)) \\
        = &\big(v_n(p_n)-v_n(\gamma(t_2))\big) - \big(v_n(p_n)-v_n(\gamma(t_1))\big)    \\
        = & -A_{h_n} \left(\gamma\vert_{[t_2,t_n]}\right) + A_{h_n} \left(\gamma\vert_{[t_1,t_n]}\right)\\
        = & A_{h_n} \left(\gamma\vert_{[t_1,t_2]}\right).
    \end{align*}
    Now because $\|p_n\|\rightarrow+\infty$, suppose $\|p_n-x\|>r$, then for each $n$, there is a $y_n=\gamma(\tau_n)$ for some $\tau_n$  such that $\|y_n-x\|=r$, 
    thus we assume $y_n\rightarrow y_r$ for some  $y_r\in\partial B(x,r)$.
    Since $\gamma_n$  is a calibrating curve, the following equation holds,
    \begin{equation*}
     v_n(x) - v_n(y_n)=  A_{h_n} \left(\gamma_n\big\vert_{[0,\tau_n]}\right) =\phi_{h_n}(x,y_n).
    \end{equation*}   
    On the otherhand, since $\phi_{h_n}$ is distance, 
    then $\phi_{h_n}(y_n,p_n)\leq\phi_{h_n}(y_n,y_r)+\phi_{h_n}(y_r,p_n)$ and $\phi_{h_n}(y_r,p_n)\leq\phi_{h_n}(p_n,y_n)+\phi_{h_n}(y_n,y_r)$
    imlplies $\vert v_n(y_n)-v_n(y_r)\vert\leq\phi_{h_n}(y_r,y_n)$, then
    we apply Theorem \ref{uniform bound for phi_hn} to have $\vert v_{p_n}(y_n)-v_{p_n}(y_r)\vert \leq\mu(\|y_n-y_r\|)\rightarrow 0$ as $n\rightarrow\infty$.
 
    On the otherhand, since $\phi_{h_n},\phi_h$ are distances in $E^N$ induced by comformal metrics and $h_n\rightarrow h$, thus $\lim_n\phi_{h_n}(x,y_r) = \phi_h(x,y_r)$,
    and by the Theorem \ref{uniform bound for phi_hn}, we have 
    \begin{equation}
        |\phi_{h_n}(x,y_r) - \phi_{h_n}(x,y_n)| \leq\phi_{h_n}(y_r,y_n) \leq\mu\|y_r-y_n\|\rightarrow 0,
    \end{equation}    
    and then
    \begin{equation}
        |\phi_h(x,y_r) - \phi_{h_n}(x,y_n)|\leq|\phi_h(x,y_r) -\phi_{h_n}(x,y_r)| + |\phi_{h_n}(x,y_r) - \phi_{h_n}(x,y_n)|\rightarrow 0,
    \end{equation}
   i.e., $\lim_n\phi_{h_n}(x,y_n) = \phi_h(x,y_r)$.
    Therefore
 \begin{align*}
     u(y_r) - u(x)&=\lim_n\bigl(v_n(x) - v_n(y_r)\bigr)\\
              &=\lim_n\bigl(v_n(x) - v_n(y_n)+v_n(y_n) - v_n(y_r)\bigr)\\
              &=\lim_n\bigl(\phi_{h_n}(x,y_n)+v_n(y_n)-v_n(y_r)\bigr)\\
              &=\phi_h(x, y_r),
 \end{align*}
 we again apply Lemma \ref{free time minimizer of phi_h exist} to take a minimizer $\gamma_r\in AC (x,y_r)$ 
 such that $ A_h (\gamma_r) = \phi_h(x,y_r) = u(y_r) - u(x)$.
 
 Now we apply Zorn's Lemma to have a maximal calibrating curve $\gamma: [0. t^*)\rightarrow E^N$, then we try to prove $t^* = +\infty$.
 We argue by contradiction. Let $t^*\neq+\infty$, $\gamma$ is a free time minimizer thus it is a true motion in $[0, t^*)$, 
 the maximal property of $t^*$ implies it is a singularity. 
 Applying von Zeipel's theorem \ref{von Zeipel} gives $I(\gamma(t))\rightarrow I^*\in[0,+\infty]$,
 If $I^*<+\infty$, Zeipel's theorem tells us $\lim_{t\rightarrow t^*}\gamma(t)$ exists in finite position, 
 then we pick another calibrating curve $\eta$ defined in some $[t^*, t^*+\sigma)$ $(\sigma>0)$ with $\eta(t^*) = \lim_{t\rightarrow t^*}\gamma(t)$, 
 concatenating $\gamma, \eta$ produces a new calibrating curve defined in $[0,t^*+\sigma)$, a contradiction.
 If $I^* = +\infty$, we choose a sequence $x_n = \gamma(t_n), (t_n\rightarrow t^{*-})$ such that $\|x_n - x\|\rightarrow+\infty$ 
 and denote $A_n =  A_h \left(\gamma\big\vert_{[0,t_n]}\right)$, then we have 
 \begin{equation}
     \|x_n - x\|^2 =\left\|\int_0^{t_n} \dot{\gamma} dt\right\|^2 = \left\|\int_0^{t_n} 1\cdot\dot{\gamma} dt\right\|^2\leq t_n\int_0^{t_n}\|\dot{\gamma}\|^2dt\leq 2t_n A_n,
 \end{equation}
 and since $\gamma$ minimizes $ A_h $, we have 
 \begin{align*}
     \frac{\|x_n - x\|^2}{2t_n} + ht_n 
     &\leq A_n + ht_n 
     = A_h (\gamma\big\vert_{[0,t_n]}) 
     =\phi(x,x_n)\\
     &\leq\left(\alpha\|x-x_n\|+\beta\|x-x_n\|^2\right)^{1/2},
 \end{align*}
 this is a contradiction since when $\|x_n - x\|$ approaching infinity, the left handside has order $O(\|x_n - x\|^2)$ 
 but is governed by the right handside whose order is $O(\|x_n - x\|)$.
 
 \end{proof}

 \begin{rem}\label{Busemann functions exist}
    The function $u$ defined in Theorem \ref{Maderna: exist calibrating curve on [0,+infty)} is called Buseman function,
    denote by $\mathcal{B}_h$ the space of these functions, like Maderna and Venturelli \cite{10.4007/annals.2020.192.2.5}. 
    They are also verified in  \cite{10.4007/annals.2020.192.2.5} to be viscosity subsolutions of the Hamilton-Jacobi equations \ref{HJ equation}. 
    We will show that it is non-empty in the proof of Theorem \ref{The set of geodesic rays is closed}. Then we can say that there exists 
    a calibrating curve starting with any prescribed initial position $x_0\in E^N$ and energy $h\geq 0$.
   \end{rem}

   \subsection{Superhyperbolic motions}\label{Superhyperbolic motions}
   
   We need to introduce an important classification of motions of $N$-body problem made by Marchal and Saari \cite{MARCHAL1976150}, 
   which helps us to achieve our main result.
   We wirte $R(t) = \max_{i<j}r_{ij}(t)$ and $r(t) = \min_{i<j}r_{ij}(t)$ which represent respectively the maximal and minimal distances between the $N$ bodies. 
   \begin{thm}\label{Classification of motions by Marchal}
       Suppose $\gamma:[0,+\infty)\rightarrow \Omega$ is an arbitrary solution of the Newton's equation, then its motion,
        no matter the sign of its energy, can only be one of the two cases:
       \begin{enumerate}
           \item $R(t)/t\rightarrow+\infty$ and $r(t)\rightarrow 0$ as $t\rightarrow+\infty$.
           \item $\gamma(t) = at + O(t^{2/3})$ for some $a\in E^N$. Here $O(t^{2/3})$ represents a term bounded by $Ct^{2/3}$ for some constant $C$.
       \end{enumerate}
   \end{thm}
   
   We need to mention that Pollard \cite{10.2307/24901766} proved $R(t)/t\rightarrow+\infty$ if and only if $r(t)\rightarrow 0$ as $t\rightarrow+\infty$,
   therefore the first case of the above theorem is valid and is called a superhyperbolic motion.
   
   The following proposition asserts that $\gamma$ is not superhyperbolic in some case, we owe it to  \cite{10.4007/annals.2020.192.2.5}.
   
    \begin{prop}\label{h-minimizer not superhyperbolic}
       For $h\geq 0$, a free time minimizer $\gamma$ of $A_h$ is not superhyperbolic
   \end{prop}
   
   \begin{proof}
      First we give a claim, that for any $x,y\in E^{N}, t>0$ and any $\gamma\in AC (x,y,t)$ defined over $[a,b]$
       with length $t$, we have
      \begin{equation*}
       \frac{\|x-y\|^2}{2t}\leq A(\gamma),
      \end{equation*}
      which is true since 
      \begin{equation*}
       \|x-y\|^2\leq\left(\int_{a}^{b}\|\dot{\gamma}\|dt\right)^2\leq t\int_{a}^{b}\|\dot{\gamma}\|^2dt,
      \end{equation*}
      then
      $$
      \frac{\|x-y\|^2}{2t}\leq\frac{1}{2}\int_{a}^{b}\|\dot{\gamma}\|^2dt \leq A(\gamma).
      $$
   
    Now suppose $\gamma$ is superhyperbolic and we take $t_{n}\rightarrow+\infty$ such that $R(t_n)/t_{n}\rightarrow+\infty$. Denote $\lambda_n = \|\gamma(t_{n})-\gamma(0)\|$,
      first notice that $R(t)\thickapprox\|\gamma(t)\|$ as $t\rightarrow+\infty$ implies $\lambda_n/t_{n}\rightarrow+\infty$.
      Then since $\gamma$ minimizes the action $ A_h $,  and $A(\gamma\vert_{[0,t_n]})+ht_n = A_h(\gamma\vert_{[0,t_n]}) =\phi_{h}(\gamma(0),\gamma(t_n))$,
      therefore by the previous assertion and Theorem \ref{uniform bound for phi_hn} we have the following inequality about $\phi_{h}(\gamma(0),\gamma(t_n))$ for any $n$:
      \begin{equation}
       \frac{\lambda_n^2}{2t_n}+ht_n\leq A(\gamma\vert_{[0,t_n]})+ht_n =\phi_{h}(\gamma(0),\gamma(t_n))\leq \left(\alpha\lambda_n + \beta\lambda_n^2 \right)^{1/2}
      \end{equation}
   for some constants $\alpha,\beta$, but this is incompatible for $\lambda_n\rightarrow+\infty$ as $t_n\rightarrow+\infty$.
   \end{proof}
   
   We have to explain here that Saari \cite{SAARI1973275} had found that there is no superhyperbolic motion in the three body problem.
   In detail, Saari gave important estimates that in three body problem, either $R(t)\approx  At$ or $ R(t)$ is bounded by $Ct^{2/3}$ as $t\rightarrow+\infty$, where $A$ and $C$ are constants.
   
   For the $N$-body problem with $N\geq 4$,  whether there is a superhyperbolic motion still needs a complete answer.
    Saari and Xia \cite{SAARI1989342} gave an example that for the collinear Newtonian four body 
   problem, the superhyperbolic motions exist in a weak sense, i.e., a motion with eslatic collision.
   For the  high-dimensional and collisionless case, there is no example at present.
   
   We have already known that geodesic rays are the second type, but the converse is not true, oscillatory motion is an example, 
   it is known that it exists in three body problem, 
   see Sitnikov \cite{Sitnikov1961TheEO}. According to Saari \cite{SAARI1973275},
   their mutual distances are bounded by $Ct^{2/3}$ for some constant $C$ hence it belongs to the second case of the above theorem,
   but its complex and weird behavior exclude the possibility of being a geodesic ray.

   \subsection{Parabolic motions starting from collison set}\label{Parabolic motions starting from collison set}
   In 2014, Luz and Maderna \cite{daluz_maderna_2014} proved the existence of parabolic solutions for any given initial position $x_0\in\Omega$, the case of starting from collision set was omitted in that paper, 
   physically speaking, when two bodies $x_i, x_j$ have the same initial position $x_i(0) = x_j(0)\in E$, their motions are ejections from their shared starting point.
   In this subsection we use the calibrating curves to give a brief proof for all $x_0\in E^N$.
   Before that we need a lemma that is needed not only in this subsection, but also in the proof of Theorem \ref{The set of geodesic rays is closed}.
   \begin{lem}\label{calibrating curve is free time minimizer}
       For $h\geq 0$, $w\in\mathcal{B}_h$ is dominated by $L+h$. For any function $u$ dominated by $L+h$, if $\gamma$  $h$-calibrates $u$, then $\gamma$ is a free time minimizer of the action potential $ A_h $.
   \end{lem}
   
   \begin{proof}
       Since $\phi_h(\cdot,\cdot)$ is a distance, for any $x,y$ in the domain of $w$, suppose 
       \begin{equation*}
        w_n = \phi(0,p_n) - \phi(x,p_n),\quad w = \lim_n w_n,
       \end{equation*}then
       \begin{equation}
           w_{p_n}(y) - w_{p_n}(x) = \phi_h(x,p_n) - \phi_h(y,p_n) \leq\phi_h\left(x, y\right),
       \end{equation}
       and since $w\in\mathcal{B}_{h}$, let $n\rightarrow+\infty$ to have $w(y) - w(x)\leq\phi_h\left(x, y\right)$.
       
       The second claim arise from $ A_h (\gamma\big\vert_{[a,b]})=u(\gamma(b))-u(\gamma(a))\leq\phi_{h}(\gamma(b),\gamma(a))\leq  A_h (\gamma\big\vert_{[a,b]})$.
   \end{proof}
   
   Now we formally give the statement of the result and its proof.
   \begin{thm}
       Given arbitrary $x_0\in E^N$, there exists a parabolic solution of Newtonian $N$-body problem having $x_0$ as its initial position.
   \end{thm}
   \begin{proof}
       According to Proposition \ref{Maderna: exist calibrating curve on [0,+infty)}, there is a calibrating curve $x(t)$ starting from $x_0\in E^N$ with energy $h=0$, 
       by Lemma \ref{calibrating curve is free time minimizer}, it is a free time minimizer, and by Proposition \ref{h-minimizer not superhyperbolic}, it is not superhyperbolic,
        hence by Theorem \ref{Classification of motions by Marchal}, we have
       $x(t) = at + O(t^{2/3})$ for some $a\in E^N$.  On the other hand Theorem \ref{free time minmizer with h geq 0, all r_ij tend to infinity} states that all mutual distances of a free time minimizer 
       goes to infinity, hence 
       $\lim_{t\rightarrow+\infty} \frac{1}{2}\|\dot{x}(t)\|^2 - U\big(x(t)\big) = \frac{1}{2}\|a\|^2$.
       On the other hand, by the conservation of total energy,
       $\frac{1}{2}\|\dot{x}(t)\|^2 - U\big(x(t)\big) \equiv 0$,
        hence $\frac{1}{2}\|a\|^2 = 0$, which imlplies $a = 0$ and then $x(t)$ is parabolic.
   \end{proof}
   
   \begin{rem}
       If the parabolic orbit $\gamma$ in the proof of the above proposition has an asymptotic configuration or limit shape, it is determined by the function $u$ which is 
       calibrated by $\gamma$.
   \end{rem}
   
   A free time minimizer with zero total energy is a parabolic motion, and conversely it is true in some sense. 
   Moeckel and Montgomery \cite{Moeckel2018} have found out in the three body problem,
   that every parabolic solution which is asymptotic to a Lagrange configuration is eventually geodesic.
   They also found that in some range of mass ratios, all geodesic rays is asymptotic to Lagrange's equilateral configuration, which is the absolute minimum of $\tilde{U}$.
   But the case of Eulerian collinear configuration is unclear.
   The above arguement combined with Maderna's main theorem in  \cite{Maderna2009} left us a question.
     
   \begin{question}
       Is a parabolic motion necessarily a geodesic ray?
   \end{question}

   \section{Proof of Theorem \ref{The set of geodesic rays is closed}}
   We give the proof of our main result in this section, before that, we need an vital technics in
   the following lemma.

\begin{lem}\label{p_n goes to infty}
    For sequence $a_n \in E^N$, suppose $\left\|a_n\right\|$ has a limit $\lambda \geq 0$, 
    there is a time sequence $t_{n}$ that goes to infinity such that the $\|\cdot\|$ norm  of
\begin{equation*}
    p_{n}:=a_{n} t_n+O\left(t_n^{2 / 3}\right),
\end{equation*}
goes to infinity. Here $O\left(t^{2 / 3}\right)$ means $O\left(t^{2 / 3}\right) \approx t^{2 / 3}$ as $t \rightarrow+\infty$.

   \end{lem}
\begin{proof}
    \begin{itemize}
        \item  If $\lambda=0$ and there is a $N^*$ such that $\left\{a_n\right\}$ vanishes for $n>N^*$, thus for any positive sequence $\left\{t_k\right\}$ that goes to infinity, we have $\left\|p_{n_k}\right\| \rightarrow+\infty$.
        \item  If $\lambda=0$ and there is a $N^*$ such that $\left\|a_n\right\|>0$ for all $n>N$, we take $t_n=\left\|a_n\right\|^{-1}$, which also makes $\left\|p_n\right\|$ and $t_n$ go to infinity as $t \rightarrow+\infty$.
        \item  If $\lambda>0$, then there exist a positive integer $N^*$ and a positive real number $L$ such that
         $\left\|a_n\right\|>L$ for all $n>N^*$, hence for any positive sequence $t_n$ that goes to infinity, $\left\|p_n\right\|=$ $\|a_n t_n+O(t_n^{2 / 3})\| \geq L t_n+O(t_n^{2 / 3}) \rightarrow+\infty$.
    \end{itemize}
\end{proof}

Now we can conduct the proof of the main Theorem \ref{The set of geodesic rays is closed}
\begin{proof}
    Recall that we have denoted by $GR\subset TE^N$ the set of initial data that lead to geodesic rays.
    Take an arbitrary sequence $\{(x_0^{(n)},v_0^{(n)})\}\subset GR$ satiesfying
     $x_{0}^{(n)}\rightarrow x_{0}\in E^N, v_{0}^{(n)}\rightarrow v_{0}\in E^{N}$ in the norm $\|\cdot\|$  as $n\rightarrow\infty$.
    We prove the motion $\gamma$ determined by $x_{0}$ and $v_{0}$ is also a geodesic ray, then any initial condition in $TE^N\setminus GR$ must have an open 
    neighberhood in $TE^N$ such that it contains no point of $GR$, and hence $TE^N\setminus GR$ is open and $GR$ is closed.  

    Based on Theorem \ref{exist of geodesic motions},
There is a sequence of free time minimizers $\gamma_{n}(t)=a_{n}t+O(t^{2/3})$ defined in $[0,+\infty)$ with intial data
$\gamma_{n}(0)=x_{0}^{(n)}\in E^N, \dot{\gamma}_{n}(0)=v_{0}^{(n)}\in E^{N}$, here $a_{n}\in E^{N}$.
 Theorem \ref{free time minmizer with h geq 0, all r_ij tend to infinity} implies  $O(t^{2/3})\approx t^{2/3}$.
  The energy $h_n$ of $\gamma_{n}$ is a constant regardless of time,
$h_{n}=\frac{1}{2}\|\dot{\gamma}_{n}(t)\|^{2}-U(\gamma(t))=\frac{1}{2}\|v_{0}^{(n)}\|^{2}-U(x_{0}^{(n)})=\frac{1}{2}\|a_n\|^2 \geq 0$,
 therefore we suppose $h_{n}\rightarrow h\geq 0, \|a_n\|\rightarrow\lambda\geq 0$ as $x_0^{(n)}, v_0^{(n)}$ approaching their limits.

Let
\begin{equation*}
    p_{n}=\gamma_{n}(t_{n})=a_{n}t_{n}+O(t_{n}^{2/3}),
\end{equation*}
By Lemma \ref{p_n goes to infty} we can assume a sequence $t_{n}$ tending to infinity such that $\left\|p_{n}\right\| \rightarrow+\infty$. 
Define $u_n(x)= \phi_{h_n}(0,p_n) - \phi_{h_n}(x,p_n)$, and by Theorem \ref{uniform bound for phi_hn}, we have
\begin{equation}\label{u_p is dominated by phi_h}
    u_n(x)-u_n(y)=\phi_{h_n}(y,p_n)-\phi_{h_n}(x,p_n)\leq\phi_{h_n}(x,y)\leq \mu(\|x-y\|) ,
\end{equation}
similarly $u_n(y)-u_n(x) \leq\phi_{h_n}(x,y)\leq \mu(\|x-y\|)  $, thus we have the inequality

\begin{equation*}
    \vert u_n(y)-u_n(x)\vert  \leq\phi_{h_n}(x,y)\leq \mu(\|x-y\|) ,
\end{equation*}
which eventually means $\{u_n\}$ is equicontinuous and uniformly bounded in any compact set of $E^{N}$ since
$\vert u_n(x)\vert \leq\left(A\|x\|+B\|x\|^2\right)^{1/2}\leq M$, $M$ is a bound depending on the compact set.
 Therefore according to Ascoli's Theorem, we can take a subsequence of $u_n$ which is also denoted by $u_n$ for convenience to define a function on $E^{N}$:
\begin{equation}\label{def of u}
    u(x)=\lim_{n\rightarrow+\infty}u_n(x),
\end{equation}
we therefore verified that $\mathcal{B}_h$ is non empty.
$u(x)$ is continuous since 
\begin{equation*}
    \vert u(x)-u(y)\vert  = \lim_n \left( u_n(x)-u_n(y) \right) \leq \mu (\|x-y\|).
\end{equation*}
Because  $\gamma_{n}$ is a free time minimizer of $ A_{h_n} $, so for arbitrary $t\geq 0$ and $t_n>t$, we have
\begin{align}\label{gamma_n calibrate over [t,t_n]}
    u_n(\gamma_{n}(t_{n})) - u_n(\gamma_{n}(t)) & = u_n(p_{n}) - u_n(\gamma_{n}(t))
    =\phi_{h_{n}}\left(\gamma_{n}(t),p_{n}\right)\\ &= A_{h_n} \left(\gamma_{n}\big\vert_{[t,t_{n}]}\right).
\end{align}
Therefore $\gamma_{n}$ $h_n$-calibrates $u_n$, 
since for any $t,t',t_n$ such that $t_{n}>t'>t\geq 0$, 
\begin{align*}
    u_n(\gamma_{n}(t'))-u_n(\gamma_{n}(t)) =& \big(u_n(\gamma_{n}(t_n))-u_n(\gamma_{n}(t))\big)\\
       & - \big( u_n(\gamma_{n}(t_n))-u_n(\gamma_{n}(t')) \big)  \\
     =& A_{h_n} \left(\gamma_{n}\big\vert_{[t,t_n]}\right) - A_{h_n} \left(\gamma_{n}\big\vert_{[t',t_n]}\right)\\
     =& A_{h_n} \left(\gamma_{n}\big\vert_{[t,t']}\right).
\end{align*}
For any $t>0$ we have 
\begin{equation}\label{u_pn is a calibrating curve of A_L+hn}
    u_n(\gamma_{n}(t))-u_n(x_{0}^{(n)}) = A_{h_n} \left(\gamma_{n}\big\vert_{[0,t]}\right).
\end{equation}

Let $\gamma:[0,t^*)\rightarrow\Omega$ be a solution with the maximal $t^*$ such that $\gamma(0)=x_{0}, \dot{\gamma}(0)=v_{0}$, 
$\gamma$ has energy constant
$$
\frac{1}{2}\|v_0\|^2-U\big((x_0)\big)=\lim_{n}\bigl(\frac{1}{2}\|v_{0}^{(n)}\|^{2}-U(x_{0}^{(n)})\bigr)=\lim_{n}h_n=h.
$$
Then we try to verify that $t^*=+\infty$  and $\gamma$  $h-$calibrates $u$, thus by Lemma \ref{calibrating curve is free time minimizer},
$\gamma$ is an free time minimizer hence a geodesic ray, and in particular by Lemma \ref{h-minimizer not superhyperbolic}, it is not superhyperbolic.

We now prove that $\gamma$ $h$-calibrates $u$. By the continuity with respect to the initial conditions, 
we have $\gamma_{n}\rightarrow\gamma$ and $\dot{\gamma}_{n}\rightarrow\dot{\gamma}$ uniformly over any compact subset of $[0,t^*)$, then
for every $t\in[0,t^*)$, let $n\rightarrow+\infty$, we have   
\begin{align*}
    \lim_{n\rightarrow+\infty} A_{h_n} \left(\gamma_{n}\vert_{[0,t]}\right) =&\lim_n \left[A_{h_n}(\gamma_n \vert_{[0,t]}) - A_h(\gamma \vert_{[0,t]})\right] + A_h(\gamma \vert_{[0,t]})\\
      = &\lim_n \left[A_{h_n}(\gamma_n \vert_{[0,t]}) - A_h(\gamma_n \vert_{[0,t]}) + A_h(\gamma_n \vert_{[0,t]}) - A_h(\gamma \vert_{[0,t]})\right]  \\
     &+ A_h(\gamma \vert_{[0,t]})\\
     =& \lim_n (h_n - h)t + \lim_n \left[A_h(\gamma_n \vert_{[0,t]}) - A_h(\gamma \vert_{[0,t]})\right] + A_h(\gamma \vert_{[0,t]})\\
      =& A_h(\gamma \vert_{[0,t]})
\end{align*}
On the other hand, the sequence $h_{n}$ is bounded since it has a limit $h\geq 0$, hence according to the inequality \ref{u_p is dominated by phi_h},
when $n\rightarrow+\infty$, we have 
\begin{equation*}
    \vert u_n(x_{0}^{(n)})-u_{p_{n}}(x_{0})\vert  \leq \phi_{h_n}(x_{0}^{(n)},x_{0}) 
    \leq\mu(\|x_{0}^{(n)}-x_{0}\|)\rightarrow 0, 
\end{equation*}
and thus
\begin{equation}
    \lim_{n\rightarrow\infty}u_n(x_{0}^{(n)})= \lim_n\left(u_n(x_{0}^{(n)})-u_n(x_{0}) + u_n(x_{0}) \right)=
     u(x_{0}),
\end{equation}
similarly
\begin{equation}
    \lim_n u_n(\gamma_{n}(t))=u(\gamma(t)).
\end{equation}
Let $n$ in the equality \ref{u_pn is a calibrating curve of A_L+hn} goes to infinity, 
 we have 
\begin{equation}
 u(\gamma(t))-u(x_{0})= A_h \left(\gamma\big\vert_{[0,t]}\right)
\end{equation}
for $t\in[0,t^*)$,
hence $\gamma$ is an $h$-calibrating curve of $u$ in its domain.

If $t^*<+\infty$, we assume $t_{n}>t^*$ for any $n$, 
Since $u\in\mathcal{B}_h$, we apply Maderna's Theorem \ref{Maderna: exist calibrating curve on [0,+infty)}, take $0<t'<t^*$, there is an $h$-calibrating curve 
$\rho:[0,+\infty)\rightarrow E^{N} $ of $u$ such that $\rho(0)=\gamma(t')$. The concatenation of $\gamma\big\vert_{[0,t']}$ with $\rho$ is also
an $h$-calibrating curve, we denote the concatenation as $\tilde{\gamma}$.
By Lemma \ref{calibrating curve is free time minimizer}, $\tilde{\gamma}$ is a global free time minimizer.
We notice that  Theorem \ref{Marchal: minimizer has no collison} implies that the concatenation $\tilde{\gamma}$ defined over $[0,+\infty)$ is the true motion and $\dot{\gamma}(t')=\dot{\rho}(0)$.
This is a contradiction since $t^*$ is finite.
Hence $\gamma$ is defined in $[0,+\infty)$ and calibrates $u$. 

On the otherhand,
for $u_n(x)-u_n(y)\leq\phi_{h_{n}}(x,y)$ from the inequality \ref{u_p is dominated by phi_h}, we let $n\rightarrow+\infty$ to have 
\begin{equation}\label{u is dominated by phi_h}
   u(x)-u(y)\leq\phi_{h}(x,y),
\end{equation}
thus by Lemma \ref{calibrating curve is free time minimizer} it is
a free time minimizer, and by Theorem \ref{ free time minimizer equivalent to geodesic ray} it is a geodesic ray. 

\end{proof}

In the above proof, the sequence $\{\|a_n\|\}$ has a limit, hence $\{a_n\}$ is bounded.
We conjecture that limit shape of $\gamma$ in the above proof is an accumulation point of the set $\{a_n\}$ in $E^N$. Therefore we have another question.

\begin{question}
    Determine the limit shape of $\gamma$, or give its location range.
\end{question}

\section{Acknowledgement}
This work is partially supported by National Natural Science Foundation of China (Grant No.12071316).

\bibliographystyle{plain}
\bibliography{ref}

\end{document}